\documentclass{amsart}
\usepackage{amsfonts}
\usepackage{amsmath}
\usepackage{amssymb}
\usepackage{amsthm}

\newtheorem{theorem}{Theorem}
\newtheorem{prop}{Proposition}
\newtheorem{lemma}{Lemma}
\newtheorem{rem}{Remark}
\newtheorem{exmp}{Example}

\begin{document}

\title{Orthogonal apartments in Hilbert Grassmannians}
\author{Mark Pankov}
\subjclass[2000]{}
\address{Department of Mathematics and Computer Science, 
University of Warmia and Mazury,
S{\l}oneczna 54, Olsztyn, Poland}
\email{pankov@matman.uwm.edu.pl}

\maketitle
\begin{abstract}
Let $H$ be an infinite-dimensional complex Hilbert space and 
let ${\mathcal L}(H)$ be the logic formed by all closed subspaces of $H$.
For every natural $k$ we denote by ${\mathcal G}_{k}(H)$
the Grassmannian consisting of $k$-dimensional subspaces.
An orthogonal apartment of ${\mathcal G}_{k}(H)$
is the set consisting of all $k$-dimensional subspaces spanned by 
subsets of a certain orthogonal base of $H$.
Orthogonal apartments can be characterized as
maximal sets of mutually compatible elements of ${\mathcal G}_{k}(H)$.
We show that every bijective transformation $f$ of ${\mathcal G}_{k}(H)$
such that $f$ and $f^{-1}$ send orthogonal apartments to orthogonal apartments
(in other words, $f$ preserves the compatibility relation in both directions)
can be uniquely extended to an automorphism of ${\mathcal L}(H)$.
\end{abstract}

\section{Introduction}
The concept of apartment comes from the theory of Tits buildings \cite{Tits},
combinatorial constructions related to groups of various types.
A {\it building} can be defined as an abstract  simplicial complex
together with a collection of subcomplexes called {\it apartments} and
satisfying some axioms.  
One of these axioms says that all apartments are isomorphic to 
the simplicial complex associated to a certain Coxeter system.
This Coxeter system defines the building type. 
The vertex set of a building is the disjoint union of 
so-called {\it building Grassmannians} \cite{Pankov-book1,Pasini}.
The intersections of apartments with a building Grassmannian are 
called apartments of this Grassmannian.
For every building of classical type all bijective transformations of Grassmannians
sending apartments to apartments can be uniquely extended to building automorphisms 
\cite{Pankov-book1}.

Every building of type $\textsf{A}_{n-1}$ is the flag complex of 
a certain $n$-dimensional vector space $V$. 
The associated Grassmannians are ${\mathcal G}_{k}(V)$, $k\in\{1,\dots,n-1\}$
formed by $k$-dimen\-sional subspaces of $V$.
Every apartment of ${\mathcal G}_{k}(V)$ is related to a certain base of $V$
and consists of all $k$-dimensional subspaces spanned by subsets of this base.
For this case the above mentioned statement on apartments preserving transformations
can be formulated as follows: 
every bijective transformation of  ${\mathcal G}_{k}(V)$
sending apartments to apartments is induced by 
a semilinear automorphism of $V$ or a semilinear isomorphism of $V$ to 
the dual vector space $V^{*}$ and the second possibility can be realized only for $n=2k$.
A more general result related to isometric embeddings of Grassmann graphs
can be found in \cite[Chapter 5]{Pankov-book2}.
Apartments preserving transformations of Grassmannians formed 
by subspaces with infinite both dimension and codimension are described in 
\cite[Theorem 3.18]{Pankov-book1}.

It is natural to ask what is happening
if we take Grassmannians of Hilbert spaces together with  {\it orthogonal apartments},
i.e. apartments defined by orthogonal bases?

Let $H$ be an infinite-dimensional complex Hilbert space.
Denote by ${\mathcal L}(H)$ the logic formed by all closed subspaces of $H$.
This logic plays an important role in mathematical foundations of quantum theory
if $H$ is separable (see, for example, \cite{Var}), 
but we will consider an arbitrary infinite-dimensional Hilbert space. 
It is well-known that every automorphism of ${\mathcal L}(H)$ 
is induced by an unitary or conjugate-unitary operator on $H$.

Recall that two closed subspaces $X,Y\subset H$ are {\it compatible} 
if there exist closed subspaces $X',Y'$ such that $X\cap Y, X',Y'$ are mutually orthogonal and 
$$X=X'+(X\cap Y),\;\;Y=Y'+(X\cap Y).$$
Theorem 2.8 from \cite{MS} can be reformulated as follows:
if $f$ is a bijective transformation of ${\mathcal L}(H)$ preserving the compatibility relation
in both directions then there exists an unitary or conjugate-unitary operator $U$
such that for every $X\in {\mathcal L}(H)$ we have $f(X)=U(X)$ or $f(X)$
coincides with the orthogonal complement of $U(X)$
(the second possibility is related with the fact that if $X$ and $Y$ compatible
then the orthogonal complement of $X$ is compatible to $Y$).
Note that
the compatibility is a property distinguishing classical logics from quantum:
in a classical logic any two elements are compatible,
a quantum logic contains non-compatible elements.

For every natural $k$ we denote by ${\mathcal G}_{k}(H)$ 
the Grassmannian formed by $k$-dimen\-sional subspaces of $H$.
In contrast to Grassmannians of vector spaces,
where for any two elements there is an apartment containing them,
two elements of ${\mathcal G}_{k}(H)$ are contained in the same orthogonal apartment
if and only if they are compatible.
Moreover, orthogonal apartments can be characterized as 
maximal sets of mutually compatible elements.
Therefore, a bijective transformation of ${\mathcal G}_{k}(H)$
preserves the class of orthogonal apartments in both directions if and only if 
it is preserving the compatibility relation in both directions.
We show that every such  transformation 
can be uniquely extended to an automorphism of the logic ${\mathcal L}(H)$.

The proof is purely combinatorial and based on 
a modification of the idea given in \cite[Section 5.2]{Pankov-book2}.
Also, we use the following fact:
every bijective transformation of ${\mathcal G}_{k}(H)$ 
preserving the orthogonality relation in both directions can be uniquely extended to
an automorphism of ${\mathcal L}(H)$ \cite{Cyory,Semrl}. 
For $k=1$ this statement is known as Wigner's theorem 
(see, for example, \cite[Theorem 4.29]{Var}).

\section{Result}
The {\it orthogonal apartment} of ${\mathcal G}_{k}(H)$
associated to an orthogonal base of $H$ 
is the set consisting of all $k$-dimensional subspaces spanned by subsets of this base. 
Every orthogonal apartment can be obtained from the unique orthonormal
base and the other related bases are formed by scalar multiples of the vectors 
from this base.

\begin{prop}\label{prop1}
The class of orthogonal apartments coincides with 
the class of maximal sets of mutually compatible elements of ${\mathcal G}_{k}(H)$.
\end{prop}

\begin{proof}
It is easy to see that every orthogonal apartment is a maximal set of 
mutually compatible elements.
Let ${\mathcal X}$ be a subset of ${\mathcal G}_{k}(H)$,
where any two elements are compatible. 
Denote by ${\mathcal X}'$ the set consisting of all minimal non-zero intersections
of elements from ${\mathcal X}$. 
Any two distinct elements of ${\mathcal X}'$ are orthogonal. 
It is clear that ${\mathcal X}$ is a maximal set of mutually compatible elements
if and only if all elements of ${\mathcal X}'$ are $1$-dimensional subspaces
and non-zero vectors lying on them form an orthogonal base of $H$.
\end{proof}

\begin{rem}\label{rem1}{\rm
Similarly, every maximal set of mutually compatible elements
of ${\mathcal L}(H)$ is formed by all closed subspaces 
spanned by subsets of a certain orthogonal base.
}\end{rem}

By Proposition \ref{prop1},
a bijective transformation of ${\mathcal G}_{k}(H)$
preserves the class of orthogonal apartments in both directions if and only if 
it is preserving the compa\-tibility relation in both directions.

\begin{theorem}\label{theorem1}
Let $f$ be a bijective transformation of ${\mathcal G}_{k}(H)$ 
such that $f$ and $f^{-1}$ send orthogonal apartments to orthogonal apartments,
in other words, $f$ preserves the compatibility relation  in both directions.
Then $f$ can be uniquely extended to an automorphism of ${\mathcal L}(H)$.
\end{theorem}

Two elements of ${\mathcal G}_{1}(H)$ are compatible if and only if they are orthogonal, i.e.
for $k=1$ our result coincides with the mentioned above Wigner's theorem.

\section{Proof}
We prove Theorem \ref{theorem1} for $k\ge 2$.
Let ${\mathcal A}$ be the orthogonal apartment of ${\mathcal G}_{k}(H)$ defined by 
an orthogonal base $\{e_{i}\}_{i\in I}$. 
For every $i\in I$ we denote by ${\mathcal A}(+i)$ and ${\mathcal A}(-i)$
the sets consisting of all elements of ${\mathcal A}$ which contain $e_{i}$ and 
do not contain $e_{i}$, respectively.
Also, we set 
$${\mathcal A}(+i,+j):={\mathcal A}(+i)\cap {\mathcal A}(+j),$$
$${\mathcal A}(+i,-j):={\mathcal A}(+i)\cap {\mathcal A}(-j),$$
$${\mathcal A}(-i,-j):={\mathcal A}(-i)\cap {\mathcal A}(-j)$$
for any distinct $i,j\in I$.

We say that a subset of ${\mathcal A}$ is {\it inexact}
if there is an orthogonal apartment distinct from ${\mathcal A}$ and containing this subset.

\begin{exmp}{\rm
For any distinct $i,j\in I$ the subset 
\begin{equation}\label{eq1}
{\mathcal A}(+i,+j)\cup {\mathcal A}(-i,-j)
\end{equation}
is inexact. 
In the base $\{e_{i}\}_{i\in I}$ we replace the vectors $e_{i},e_{j}$
by any other pair of orthogonal vectors belonging to 
the $2$-dimensional subspace ${\mathbb C}e_{i}+{\mathbb C}e_{j}$
and consider the associated orthogonal apartment ${\mathcal A}'$.
Then
$${\mathcal A}\cap {\mathcal A}'={\mathcal A}(+i,+j)\cup {\mathcal A}(-i,-j).$$
}\end{exmp}

\begin{lemma}\label{lemma1}
Every inexact subset of ${\mathcal A}$ 
is contained in a subset of type \eqref{eq1}, i.e.
every maximal inexact subset is of type \eqref{eq1}.
\end{lemma}

\begin{proof}
Let ${\mathcal X}$ be an inexact subset of ${\mathcal A}$.
For every $i\in I$ we denote by $S_{i}$ the intersection of all subspaces $X$ 
which satisfy one of the following conditions:
\begin{enumerate}
\item[$\bullet$]  
$X$ is an element of ${\mathcal X}$ containing $e_{i}$,
\item[$\bullet$]  
$X$ is the orthogonal complement of an element from ${\mathcal X}$ non-containing $e_{i}$. 
\end{enumerate}
Each $S_i$ is non-zero.
If every $S_{i}$ is $1$-dimensional then 
${\mathcal A}$ is the unique orthogonal apartment containing ${\mathcal X}$
which contradicts the fact that ${\mathcal X}$ is inexact.
So, there is at least one $i\in I$ such that $\dim S_{i}\ge 2$.
We take any $j\ne i$ such that $e_{j}$ belongs to $S_{i}$. 
Then 
$${\mathcal X}\subset {\mathcal A}(+i,+j)\cup {\mathcal A}(-i,-j).$$
Indeed, if $X\in {\mathcal X}$ contains $e_{i}$ then $e_{j}\in S_{i}\subset X$
and $X$ belongs to ${\mathcal A}(+i,+j)$. 
If $X\in {\mathcal X}$ does not contain $e_{i}$ then $e_{j}\in S_{i}\subset X^{\perp}$
which means that $e_{j}$ is not contained in $X$
and $X$ belongs to ${\mathcal A}(-i,-j)$. 
\end{proof}

A subset ${\mathcal C}\subset {\mathcal A}$ is said to be {\it complementary}
if ${\mathcal A}\setminus {\mathcal C}$ is a maximal inexact subset, i.e.
$${\mathcal A}\setminus {\mathcal C}={\mathcal A}(+i,+j)\cup {\mathcal A}(-i,-j)$$
for some distinct $i,j\in I$.
Then 
$${\mathcal C}={\mathcal A}(+i,-j)\cup {\mathcal A}(+j,-i).$$
This complementary subset will be denoted by ${\mathcal C}_{ij}$.
Note that ${\mathcal C}_{ij}={\mathcal C}_{ji}$.

\begin{lemma}\label{lemma2}
Two $k$-dimensional subspaces $X,Y\in{\mathcal A}$ are orthogonal if and only if 
the number of complementary subsets of ${\mathcal A}$ 
containing both $X$ and $Y$ is finite.
\end{lemma}

\begin{proof}
If the complementary subset ${\mathcal C}_{ij}$ contains both $X$ and $Y$ then 
one of the following possibilities is realized:
\begin{enumerate}
\item[(1)] one of $e_{i},e_{j}$ belongs to $X\setminus Y$ and 
the other to $Y\setminus X$,
\item[(2)] one of $e_{i},e_{j}$ belongs to $X\cap Y$ and the other is not contained in $X+Y$.
\end{enumerate}
The number of complementary subsets ${\mathcal C}_{ij}$ satisfying (1) is finite. 
If $X$ and $Y$ are orthogonal then $X\cap Y=0$ and
there is no ${\mathcal C}_{ij}$ satisfying (2).
In the case when $X\cap Y\ne 0$, there are infinitely many such ${\mathcal C}_{ij}$. 
\end{proof}

Let $f$ be a bijective transformation of ${\mathcal G}_{k}(H)$
such that $f$ and $f^{-1}$ send orthogonal apartments to orthogonal apartments.
For any orthogonal $k$-dimensional subspaces $X,Y\subset H$
there is an orthogonal apartment ${\mathcal A}\subset {\mathcal G}_{k}(H)$ containing them.
It is clear that $f$ sends inexact subsets of ${\mathcal A}$ 
to inexact subsets of $f({\mathcal A})$.
Similarly, $f^{-1}$ transfers inexact subsets of $f({\mathcal A})$ to inexact subsets of ${\mathcal A}$. 
This implies that  ${\mathcal X}$ is a maximal inexact subset of ${\mathcal A}$
if and only if $f({\mathcal X})$ is a maximal inexact subset of $f({\mathcal A})$.
Therefore, a subset of ${\mathcal A}$ is complementary if and only if 
its image is a complementary subset of $f({\mathcal A})$.
Then Lemma \ref{lemma2} guarantees that  $f(X)$ and $f(Y)$ are orthogonal.
Similarly, we establish that $f^{-1}$ transfers orthogonal subspaces to orthogonal subspaces.  
So, $f$ preserves the orthogonality relation in both directions 
which means that it can be uniquely extended to an automorphism of ${\mathcal L}(H)$.

\section{Final remarks}
Consider the case when $H$ is a complex Hilbert space of finite dimension $n$.
As above, we write ${\mathcal G}_{k}(H)$ for the Grassmannian of $k$-dimensional subspaces of $H$.
Let ${\mathcal A}$ be an ortho\-go\-nal apartment of ${\mathcal G}_{k}(H)$.
In almost all cases, the dimension of the intersection of two elements from ${\mathcal A}$
can be characterized in terms of complementary subsets and there is an analogue of Theorem \ref{theorem1}
(we get a statement similar to \cite[Theorem 2.8]{MS} if $n=2k$).
For $n=2k\pm 2$ such a characterization is impossible 
and we need some additional arguments. 

For example, if $n=6$ and $k=2$ then $|{\mathcal A}|=15$ and 
the intersection of two distinct elements from ${\mathcal G}_{k}(H)$ is zero or $1$-dimensional. 
Any pair of elements from ${\mathcal A}$ is contained 
in precisely $4$ distinct complementary subsets. 
The intersection of two complementary subsets always is a $3$-element subset
and we cannot distinguish one pair of complementary subsets from others.

\end{document}